\documentclass[11pt]{article}
\usepackage{amsfonts}
\usepackage{mathrsfs}
\usepackage{bbm}
\usepackage{amsfonts}
\usepackage{amssymb,amsmath,amsthm,graphicx}
\usepackage{float}
\usepackage[colorlinks, linkcolor=red]{hyperref}
\usepackage{color, xcolor}
\usepackage{cite}
\usepackage[paper=a4paper, left=1.6cm, right=1.6cm, top=1.8cm, bottom=1.6cm, headheight=5.5pt, footskip=0.8cm, footnotesep=0.8cm, centering, includefoot]{geometry}

\hypersetup{urlcolor=red, citecolor=red}

\DeclareMathOperator{\divv}{div}
\DeclareMathOperator{\curl}{curl}

\allowdisplaybreaks

\begin{document}
\title{Global strong solution for 3D compressible heat-conducting
magnetohydrodynamic equations revisited
\thanks{
This research was partially supported by National Natural Science Foundation of China (Nos. 11901288, 11901474, 12071359), Scientific Research Foundation of Jilin Provincial Education Department (No. JJKH20210873KJ), Postdoctoral Science Foundation of China (No. 2021M691219), and Exceptional Young Talents Project of Chongqing Talent (No. cstc2021ycjh-bgzxm0153).}
}

\author{Yang Liu$\,^{\rm 1, 2}\,$\quad Xin Zhong$\,^{\rm 3}\,${\thanks{Corresponding author. E-mail address:
liuyang0405@ccsfu.edu.cn (Y. Liu), xzhong1014@amss.ac.cn (X. Zhong).}}
\date{}\\
\footnotesize $^{\rm 1}\,$
School of Mathematics, Jilin University, Changchun 130012, P. R. China\\
\footnotesize $^{\rm 2}\,$ College of Mathematics, Changchun Normal
University, Changchun 130032, P. R. China\\
\footnotesize $^{\rm 3}\,$ School of Mathematics and Statistics, Southwest University, Chongqing 400715, P. R. China} \maketitle
\newtheorem{theorem}{Theorem}[section]
\newtheorem{definition}{Definition}[section]
\newtheorem{lemma}{Lemma}[section]
\newtheorem{proposition}{Proposition}[section]
\newtheorem{corollary}{Corollary}[section]
\newtheorem{remark}{Remark}[section]
\renewcommand{\theequation}{\thesection.\arabic{equation}}
\catcode`@=11 \@addtoreset{equation}{section} \catcode`@=12
\maketitle{}

\begin{abstract}
We revisit the 3D Cauchy problem of compressible heat-conducting magnetohydrodynamic equations with vacuum as far field density. By delicate energy method, we derive global existence and uniqueness of strong solutions provided that $(\|\rho_0\|_{L^\infty}+1)
\big[\|\rho_0\|_{L^3}+(\|\rho_0\|_{L^\infty}+1)^2\big(\|\sqrt{\rho_0}u_0\|_{L^2}^2
+\|b_0\|_{L^2}^2\big)\big]\big[\|\nabla u_0\|_{L^2}^2+(\|\rho_0\|_{L^\infty}+1)\big(\|\sqrt{\rho_0}E_0\|_{L^2}^2+\|\nabla b_0\|_{L^2}^2\big)\big]$ is properly small. In particular, the smallness condition is independent of any norms of the initial data. This work improves our previous results \cite{LZ20,LZ21}.
\end{abstract}



%%%%%%%%%%%%%%%%%%%%%%%%%%%%%%%%%%%%%%%%%%%%%%%%%%%%%%%%%%%%%%%%%%%%%%%%%%%%%%%%%%%%%%%%%%%%%%%%%%

\section{Introduction}
Let $\Omega\subseteq\mathbb{R}^3$ be a domain, the motion of a viscous, compressible, and heat conducting magnetohydrodynamic (MHD) flow in $\Omega$ can be described by full compressible MHD equations (see \cite[Chapter 3]{LQ2012}):
\begin{align}\label{a1}
\begin{cases}
\rho_t+\divv(\rho u)=0,\\
\rho u_t+\rho u\cdot\nabla u-\mu\Delta u-(\lambda+\mu)\nabla\divv u+\nabla p=\curl b\times b,\\
c_v\rho(\theta_t+u\cdot\nabla\theta)+p\divv u-\kappa\Delta\theta=\mathcal{Q}(\nabla u)
+\nu|\curl b|^2,\\
b_t-b\cdot\nabla u+u\cdot\nabla b+b\divv u=\nu\Delta b,\\
\divv b=0,
\end{cases}
\end{align}
where the unknowns $\rho\ge 0$, $u\in\mathbb{R}^3$, $\theta\ge 0$, and $b\in\mathbb{R}^3$ are
the density, velocity, pressure, absolute temperature, and magnetic field,
respectively; $p=R\rho\theta$, with positive constant $R$, is the pressure, and
\begin{align}
\mathcal{Q}(\nabla u)=\frac{\mu}{2}|\nabla u+(\nabla u)^\top|^2
+\lambda(\divv u)^2,
\end{align}
with $(\nabla u)^\top$ being the transpose of $\nabla u$. The constant viscosity coefficients $\mu$ and $\lambda$ satisfy the physical restrictions
\begin{align}\label{1.2}
\mu>0, \quad 2\mu+3\lambda\ge 0.
\end{align}
Positive constants $c_\nu$, $\kappa$, and $\nu$ are the heat capacity, the ratio of the heat conductivity coefficient over the heat capacity, and the magnetic diffusive coefficient, respectively.

In this paper, we will consider the Cauchy problem for \eqref{a1} in $\mathbb{R}^3\times(0,T)$ with the initial condition
\begin{align}\label{a2}
(\rho, \rho u,\rho\theta, b)(x, 0)=(\rho_0, \rho_0u_0, \rho_0\theta_0, b_0)(x), \quad x\in\mathbb{R}^3,
\end{align}
and the far field behavior
\begin{align}\label{a3}
(\rho, u, \theta, b)(x, t)\rightarrow (0, 0, 0, 0)~{\rm as}~|x|\rightarrow \infty.
\end{align}

It should be noted that the system \eqref{a1} becomes the compressible non-isentropic Navier-Stokes equations when there is no electromagnetic field,  which is one of the most important systems in fluid dynamics.
With the assumption that the initial data differ only slightly from the equilibrium values (constants), Matsumura and Nishida \cite{MN80,MN83} first proved the global existence of smooth solutions to initial boundary value problems and the Cauchy problem.
Feireisl \cite{F2004} obtained the global existence of
the so-called ``variational solutions" in the sense that the energy equation is replaced by an energy inequality. Recently, Huang and Li \cite{HL18} derived global well-posedness of strong solutions to the full compressible
Navier-Stokes equations in $\mathbb{R}^3$ with non-vacuum at infinity which are of small energy but possibly large oscillations.
Wen and Zhu \cite{WZ17} showed global existence of strong solutions with far-field vacuum under the condition that the initial mass is properly small in certain sense. Meanwhile, Li \cite{L20} obtained a new type
of global strong solutions under some smallness condition on the scaling invariant quantity. Very recently, Liang \cite{L21} established global strong solutions when the initial energy is small. Moreover, decay rates were also determined.

Let's turn our attention to the full compressible MHD equations \eqref{a1}.
As a couple system, \eqref{a1} contains much richer structures than the full compressible Navier-Stokes equations. It is not merely a combination of fluid equations and magnetic field equations but an interactive system. Their distinctive features make analytical studies a great challenge but offer new opportunities.
For the initial data satisfying some compatibility condition, Fan and Yu \cite{FY09} established the local existence and uniqueness of strong solutions to the problem \eqref{a1}--\eqref{a3}. Later, Liu and Zhong \cite{LZ20} extended this local existence result to be a global one provided that $\|\rho_0\|_{L^\infty}+\|b_0\|_{L^3}$ is suitably small and the viscosity coefficients satisfy $3\mu>\lambda$. This result was improved in \cite{LZ21} where the authors proved the global existence and uniqueness of strong solutions, which may be of possibly large oscillations, provided that the initial data are of small total energy. At the same time, Hou-Jiang-Peng \cite{HJP22} obtained global strong solutions under the condition that $\|\rho_0\|_{L^1}+\|b_0\|_{L^2}$ is suitably small.
For the global existence of weak solutions, we refer to \cite{HW08,DF2006,LG2014,LS2021}
and references therein. There are also some interesting mathematical results concerning the global existence of (weak, strong or classical) solutions to the compressible isentropic MHD equations, please refer to \cite{FL2020,SH12,HHPZ,HW10,LXZ13,WW17,LSX16}.

The aim of the present paper is to extend the result in \cite{L20} to the compressible heat-conductive magnetohydrodynamic flows.
This is a nontrivial generalization, since one has to control the strong nonlinear terms involved with the magnetic field. Furthermore, the restriction on the viscosity coefficients is relaxed. These are exactly the new points of this paper.

For $1\le p\le \infty$ and integer $k\ge 0$, we denote the standard homogeneous and inhomogeneous
Sobolev spaces as follows:
\begin{align*}
\begin{cases}
L^p=L^p(\mathbb{R}^3),~ ~ W^{k, p}=L^p\cap D^{k, p},~~ H^k=W^{k, 2},\\
D^{k, p}=\{u\in L_{loc}^1(\mathbb{R}^3): \|\nabla^ku\|_{L^p}<\infty\}, D^k=D^{k, 2},\\
D_0^1=\{u\in L^6(\mathbb{R}^3): \|\nabla u\|_{L^2}<\infty\}.
\end{cases}
\end{align*}

We can now state our main result.
\begin{theorem}\label{thm1}
Assume that $3\mu>\lambda$ and let $q\in (3, 6]$ be a fixed constant. Let the initial data $(\rho_0\ge 0, u_0, \theta_0\ge 0, b_0)$ satisfy
\begin{gather}
\rho_0\in H^1\cap W^{1, q},~ (\sqrt{\rho_0}u_0,\sqrt{\rho_0}\theta_0)\in L^2, ~(u_0, \theta_0)\in D_0^1\cap D^2,~ b_0\in H^2,
\end{gather}
and the compatibility condition
\begin{align}
\begin{cases}
-\mu\Delta u_0-(\lambda+\mu)\nabla{\rm div} u_0+\nabla(R\rho_0\theta_0)
-\curl b_0\times b_0=\sqrt{\rho_0}g_1,\\
-\kappa\Delta\theta_0-\mathcal{Q}(\nabla u_0)-\eta|\curl b_0|^2=\sqrt{\rho_0}g_2,\label{qqw}
\end{cases}
\end{align}
with $g_1,g_2\in L^2(\mathbb{R}^3)$. There exists a small positive
constant $\varepsilon_0$ depending only on $R$,  $\mu$, $\lambda$, $\nu$, $\kappa$, and $c_v$
 such that if
\begin{align}\label{1.8}
N_0\triangleq \bar{\rho}\big[\|\rho_0\|_{L^3}+\bar{\rho}^2(\|\sqrt{\rho_0}u_0\|_{L^2}^2
+\|b_0\|_{L^2}^2)\big]
\big[\|\nabla u_0\|_{L^2}^2+\bar{\rho}(\|\sqrt{\rho_0}E_0\|_{L^2}^2+\|\nabla b_0\|_{L^2}^2)\big]\le \varepsilon_0,
\end{align}
where $E_0=\frac{|u_0|^2}{2}+c_v\theta_0$ and $\bar{\rho}=\|\rho_0\|_{L^\infty}+1$,
then the problem \eqref{a1}--\eqref{a3} has a unique global strong solution $(\rho, u, \theta, b)$ satisfying
\begin{align*}
\begin{cases}
\rho\in C([0, T]; H^1\cap W^{1, q}),\ \rho_t\in C([0, T]; L^2\cap L^q),\\
(u, \theta)\in C([0, T]; D^1\cap D^2)\cap L^2(0, T; D^{2, q}),\ (u_t, \theta_t)\in L^2(0, T; D_0^1), \\
b\in C([0, T]; H^2)\cap L^2(0, T; H^3),\ b_t\in C([0, T];L^2)\cap L^2(0, T; H^1),\\
(\sqrt{\rho}u_t, \sqrt{\rho}\theta_t)\in L^\infty(0, T; L^2).
\end{cases}
\end{align*}
\end{theorem}

\begin{remark}
It should be noted that our smallness assumption is independent of any norms of the initial data, which is in sharp contrast to \cite{LZ20,LZ21,HJP22}
where they established global strong solution under some smallness conditions depending on the initial data.
\end{remark}

\begin{remark}
It is not hard to check that the quantity $N_0$ in \eqref{1.8} is scaling invariable
under the transform
\begin{align*}
\rho_\lambda(x, t)=\rho(\lambda x, \lambda^2t), ~u_\lambda(x, t)=\lambda u(\lambda x, \lambda^2t),
~\theta_\lambda(x, t)=\lambda^2(\lambda x, \lambda^2t),~b_\lambda(x, t)=\lambda b(\lambda x, \lambda^2t).
\end{align*}
Thus we generalize the result in \cite{L20} to the compressible non-isentropic MHD equations. Nevertheless, we should point out that the magnetic field acts some significant roles since $N_0$ differs slightly from that of in \cite{L20} when $b\equiv b_0\equiv0$. Moreover, the restriction $2\mu>\lambda$ in \cite{L20} on the viscosity
coefficients is relaxed to $3\mu>\lambda$.
\end{remark}

It seems that the methods used in \cite{LZ20,LZ21,HJP22} are not available here because our smallness condition depends only on the parameters in the system. We mainly apply some techniques developed by Li \cite{L20} to give the proof of Theorem \ref{thm1}. However, compared with compressible heat-conducting Navier-Stokes equations considered in \cite{L20}, due to the strong coupling and interplay interaction between the fluid motion and the magnetic field, the crucial techniques of proofs in \cite{L20} cannot be adapted directly.
%Moreover, unlike \cite{LZ20}, it is impossible to get the smallness of $\|H\|_{L^3}$ which plays a critical role in dealing with the coupling terms $\curl(u\times H)$ and $\curl H\times H$.
To overcome these difficulties, we first obtain
$L^\infty(0,T;L^4)$ bound of $b$ in terms of $L^\infty_tL^2_x$-norm of $b,\nabla b$, and $\nabla u$ (see \eqref{2.4}).
Then some necessary lower order time-independent estimates are obtained (see \eqref{2.8} and \eqref{2.14}).
Next, with the help of the effect viscous flux, we derive $L^\infty(0,T;L^2)$ estimate of $\nabla u$ under \textit{a priori hypothesis} \eqref{2.23} (see Lemma \ref{lem}). So, the next key step is to complete the proof of the \textit{a priori hypothesis}, that is, to show that $\|\rho(t)\|_{L^\infty}$ is in fact strictly less than $4\bar\rho$. Inspired by \cite[Proposition 2.6]{L20}, we find that $\|\rho\|_{L^\infty(0,T;L^\infty)}$ is indeed bounded by other norms of the solution and could be as small as
desired under smallness condition on $N_0$ (see Proposition \ref{pro}). This in particular completes the proof of \textit{a priori hypothesis}. Having the time-independent estimates at hand, we can show Theorem \ref{thm1} by continuity arguments as those in \cite{LZ20}.

The rest of the paper is organized as follows. In Section 2, we recall some known facts and elementary inequalities which will be used later. Section 3 is devoted to the proof of Theorem \ref{thm1}.

\section{Preliminaries}
In this section, we recall some known facts and elementary inequalities which will be used frequently later.

We begin with the local existence of a unique strong solution with vacuum to the problem \eqref{a1}--\eqref{a3}, whose proof can be obtained by similar ways as those in \cite{FY09}.

\begin{lemma}
Assume that the initial data $(\rho_0\ge 0, u_0, \theta_0\ge 0, b_0)$ satisfies the conditions in Theorem \ref{thm1}. Then, there exists
a positive time $T_0>0$ depending only on $R$, $\mu$, $\lambda$, and $\psi_0$, such that the problem \eqref{a1}--\eqref{a3} admits a unique strong
solution in $\mathbb{R}^3\times (0, T_0]$, where $\psi_0$ is a positive constant such that
\begin{align*}
\|\rho_0\|_{H^1\cap W^{1, q}}+\|(u_0, \theta_0)\|_{D_0^1\cap D^2}+\|b_0\|_{H^2}+\|(\sqrt{\rho_0}u_0, \sqrt{\rho_0}\theta_0, g_1, g_2)\|_{L^2}\le \psi_0.
\end{align*}
\end{lemma}

Next, the following well-known Gagliardo-Nirenberg inequality (see \cite[Chapter 10, Theorem 1.1]{D2016}) will be
used frequently later.
\begin{lemma}\label{l22}
Assume that $f\in D^{1, m}\cap L^r$ with $m,r\geq1$, then there exists a constant $C$ depending only on $q$, $m$, and $r$ such that
\begin{align*}
\|f\|_{L^q}\le C\|\nabla f\|_{L^m}^\vartheta\|f\|_{L^r}^{1-\vartheta},
\end{align*}
where $\vartheta=\big(\frac{1}{r}-\frac{1}{q}\big)/
\big(\frac{1}{r}-\frac{1}{m}+\frac13\big)$ and the admissible range of $q$ is the following:
\begin{itemize}
\item if $m<3$, then $q$ is between $r$ and $\frac{3m}{3-m}$;
\item if $m=3$, then $q\in [r, \infty)$;
\item if $m>3$, then $q\in [r, \infty]$.
\end{itemize}
\end{lemma}

\section{Proof of Theorem \ref{thm1}}

\subsection{\textit{A priori} estimates}
In this subsection, we will establish some necessary \textit{a priori} bounds for smooth solutions to the Cauchy problem \eqref{a1}--\eqref{a3}.
In what follows, we always assume that $(\rho, u, \theta, b)$
is a strong solution to the problem \eqref{a1}--\eqref{a3} in
$\mathbb{R}^3\times(0, T]$ for some positive time $T$.
Meanwhile, $C, C_i$, and $c_i\ (i=1, 2, \cdots)$ denote generic positive
constants which rely only on $R$,  $\mu$, $\lambda$, $\nu$, $\kappa$, and $c_v$. For simplicity, in what follows, we write
\begin{align*}
\int\cdot dx=\int_{\mathbb{R}^3}\cdot dx, \quad \nu=\kappa=c_v=1.
\end{align*}
\begin{lemma}\label{l31}
It holds that
\begin{align}\label{2.2}
&\sup_{0\le t\le T}(\|\sqrt{\rho}u\|_{L^2}^2+\|b\|_{L^2}^2)+
\int_0^T(\|\nabla u\|_{L^2}^2+\|\nabla b\|_{L^2}^2)dt\nonumber\\
&\le \|\sqrt{\rho_0}u_0\|_{L^2}^2+\|b_0\|_{L^2}^2+C\int_0^T\|\rho\|_{L^3}^2\|\nabla\theta\|_{L^2}^2dt.
\end{align}
\end{lemma}
\begin{proof}[Proof]
Multiplying $\eqref{a1}_2$ by $u$ and $\eqref{a1}_4$ by $b$, respectively,
then adding the two resulting equations together and integrating over $\mathbb{R}^3$, we obtain from
$\mu+\lambda>0$\footnote{From \eqref{1.2}, we have $3\mu+3\lambda>0$. Thus the result follows.}, H{\"o}lder's inequality, and Sobolev's inequality that
\begin{align*}
& \frac12\frac{d}{dt}(\|\sqrt{\rho}u\|_{L^2}^2+\|b\|_{L^2}^2)+\mu\|\nabla u\|_{L^2}^2+(\mu+\lambda)\|\divv u\|_{L^2}^2+\|\nabla b\|_{L^2}\nonumber\\
& =R\int \rho\theta\divv udx
\le R\|\rho\|_{L^3}\|\theta\|_{L^6}\|\divv u\|_{L^2}
\le \frac{\mu+\lambda}{2}\|\divv u\|_{L^2}^2
+C\|\rho\|_{L^3}^2\|\nabla\theta\|_{L^2}^2,
\end{align*}
which implies that
\begin{align}\label{p11}
\frac{d}{dt}(\|\sqrt{\rho}u\|_{L^2}^2+\|b\|_{L^2}^2)+\|\nabla u\|_{L^2}^2+\|\nabla b\|_{L^2}^2
\le C\|\rho\|_{L^3}^2\|\nabla\theta\|_{L^2}^2.
\end{align}
Hence, integrating \eqref{p11} over $[0,T]$ leads to \eqref{2.2}.
\end{proof}

\begin{lemma}
It holds that
\begin{align}\label{2.4}
&\sup_{0\le t\le T}\big(\|\nabla b\|_{L^2}^2+\|b\|_{L^4}^4\big)+\int_0^T\big(\|b_t\|_{L^2}^2+\|\nabla^2b\|_{L^2}^2
+\||b||\nabla b|\|_{L^2}^2\big)dt\nonumber\\
&\le \|\nabla b_0\|_{L^2}^2+\|b_0\|_{L^4}^4
+C\sup_{0\le t\le T}\big(\|b\|_{L^2}^2\|\nabla b\|_{L^2}^2\big)^\frac32\int_0^T\|\nabla^2b\|_{L^2}^2dt\nonumber\\
&\quad+C\sup_{0\le t\le T}\big(\|b\|_{L^2}^2\|\nabla u\|_{L^2}^2\big)\int_0^T\|\nabla u\|_{L^2}^6dt.
\end{align}
\end{lemma}
\begin{proof}[Proof]
1. We deduce from $\eqref{a1}_4$ that
\begin{align}\label{2.5}
\frac{d}{dt}\|\nabla b\|_{L^2}^2+\|b_t\|_{L^2}^2+\|\nabla^2b\|_{L^2}^2
&=\int|b_t-\Delta b|^2dx=\int|b\cdot\nabla u-u\cdot\nabla b
-b\divv u|^2dx\nonumber\\
&\le C\|\nabla u\|_{L^2}^2\|b\|_{L^\infty}^2+C\|u\|_{L^6}^2\|\nabla b\|_{L^2}\|\nabla b\|_{L^6}\nonumber\\
&\le C\|b\|_{L^2}^\frac12\|\nabla^2b\|_{L^2}^\frac32\|\nabla u\|_{L^2}^2\nonumber\\
&\le \frac18\|\nabla^2b\|_{L^2}^2+C(\|b\|_{L^2}^2\|\nabla u\|_{L^2}^2)\|\nabla u\|_{L^2}^6,
\end{align}
where we have used the following Gagliardo-Nirenberg inequality
\begin{align}\label{2.6}
\|\nabla v\|_{L^2}\le C\|v\|_{L^2}^\frac12\|\nabla^2v\|_{L^2}^\frac12,
\ \|v\|_{L^\infty}\le C\|v\|_{L^2}^\frac14\|\nabla^2v\|_{L^2}^\frac34.
\end{align}

2. Multiplying $\eqref{a1}_4$ by $4|b|^2b$ and integration by parts, we get from \eqref{2.6} and Sobolev's inequality that
\begin{align*}
&\frac{d}{dt}\int|b|^4dx+4\int(|\nabla b|^2|b|^2+2|\nabla|b||^2|b|^2)dx\nonumber\\
&=4\int(b\cdot\nabla u-u\cdot\nabla b-b\divv u)|b|^2bdx\nonumber\\
&=4\int b\cdot\nabla u\cdot b|b|^2dx-3\int|b|^4\divv udx\nonumber\\
&\le C\int|b|^2|\nabla u|^2dx+C\int|b|^6dx\nonumber\\
&\le C\|b\|_{L^\infty}^2\|\nabla u\|_{L^2}^2+C\|b\|_{L^6}^2\||b|^2\|_{L^6}\||b|^2\|_{L^2}\nonumber\\
&\le C\|b\|_{L^2}^\frac12\|\nabla^2b\|_{L^2}^\frac32\|\nabla u\|_{L^2}^2
+C\||b||\nabla b|\|_{L^2}\|b\|_{L^4}^2\|\nabla b\|_{L^2}^2\nonumber\\
&\le \frac18\|\nabla^2 b\|_{L^2}^2
+C(\|b\|_{L^2}^2\|\nabla u\|_{L^2}^2)\|\nabla u\|_{L^2}^6
+C\||b||\nabla b|\|_{L^2}\|b\|_{L^3}\|\nabla b\|_{L^2}\|b\|_{L^2}\|\nabla^2b\|_{L^2}\nonumber\\
&\le \frac18\|\nabla^2 b\|_{L^2}^2
+C(\|b\|_{L^2}^2\|\nabla u\|_{L^2}^2)\|\nabla u\|_{L^2}^6
+C\||b||\nabla b|\|_{L^2}\|b\|_{L^2}^\frac32\|\nabla b\|_{L^2}^\frac32\|\nabla^2b\|_{L^2}\nonumber\\
&\le \frac18\|\nabla^2b\|_{L^2}^2+\frac12\||b||\nabla b|\|_{L^2}^2
+C\|b\|_{L^2}^3\|\nabla b\|_{L^2}^3\|\nabla^2b\|_{L^2}^2
+C(\|b\|_{L^2}^2\|\nabla u\|_{L^2}^2)\|\nabla u\|_{L^2}^6.
\end{align*}
This along with \eqref{2.5} yields that
\begin{align}\label{2.7}
&\frac{d}{dt}\big(\|\nabla b\|_{L^2}^2+\|b\|_{L^4}^4\big)+\|b_t\|_{L^2}^2+\|\nabla^2b\|_{L^2}^2+\||b||\nabla b|\|_{L^2}^2\nonumber\\
&\le C\big(\|b\|_{L^2}^2\|\nabla b\|_{L^2}^2\big)^\frac32\|\nabla^2b\|_{L^2}^2
+C\big(\|b\|_{L^2}^2\|\nabla u\|_{L^2}^2\big)\|\nabla u\|_{L^2}^6.
\end{align}
Thus, \eqref{2.4} follows from \eqref{2.7} integrated in $t$ over $[0,T]$.
\end{proof}

\begin{lemma}
Assume that $3\mu>\lambda$, it holds that
\begin{align}\label{2.8}
&\sup_{0\le t\le T}\|\sqrt{\rho}E\|_{L^2}^2
+\int_0^T\big(\||u||\nabla u|\|_{L^2}^2+\|\nabla\theta\|_{L^2}^2\big)dt\nonumber\\
&\le C\|\sqrt{\rho_0}E_0\|_{L^2}^2
+C\sup_{0\le t\le T}\big(\|b\|_{L^2}\|\nabla b\|_{L^2}\big)\int_0^T\|\nabla^2 b\|_{L^2}^2dt\nonumber\\
&\quad+C\sup_{0\le t\le T}\big(\|b\|_{L^2}^2\|\nabla b\|_{L^2}^2\big)\int_0^T\|\nabla u\|_{L^2}^6dt+C\sup_{0\le t\le T}\|\nabla b\|_{L^2}^4\int_0^T\|\nabla u\|_{L^2}^2dt\nonumber\\
&\quad+C\int_0^T\|\sqrt{\rho}\theta\|_{L^2}\|\nabla\theta\|_{L^2}\||u||\nabla u|\|_{L^2}\|\rho\|_{L^3}^\frac{1}{2}dt.
\end{align}
\end{lemma}
\begin{proof}[Proof]
1. For $E=\frac{|u|^2}{2}+c_v\theta$, we infer from \eqref{a1} that
\begin{align}\label{b5}
\rho (E_t+ u\cdot\nabla E)+\divv(up)-\Delta\theta=\divv(\mathcal{S}\cdot u)+\curl b\times b
+|\curl b|^2,
\end{align}
where $\mathcal{S}=\mu(\nabla u+(\nabla u)^\top)+\lambda\divv u\mathbb{I}_3$ with $\mathbb{I}_3$ being the identity matrix of order $3$.
Multiplying \eqref{b5} by $E$ and integrating the resultant
over $\mathbb{R}^3$, it follows from integration by parts and Young's inequality that
\begin{align}\label{2.11}
\frac{1}{2}\frac{d}{dt}\|\sqrt{\rho}E\|_{L^2}^2+\|\nabla\theta\|_{L^2}^2
&\le -\frac{1}{2}\int\nabla\theta\cdot\nabla|u|^2dx+\int(up-\mathcal{S}\cdot u)\cdot \nabla Edx\nonumber\\
&\quad+C\int\big(|u||b|^2|\nabla E|+|\nabla u||b|^2E\big)dx
+\int |\curl b|^2Edx\nonumber\\
&\le \frac{1}{6}\|\nabla\theta\|_{L^2}^2+\frac{3}{8}\||u||\nabla u|\|_{L^2}^2+C\int\rho^2\theta^2|u|^2dx\nonumber\\
&\quad+C\int\big(|u||b|^2|\nabla E|+|\nabla u||b|^2E\big)dx+C\int|\nabla E||\nabla b||b|dx\nonumber\\
&\quad+C\int|E||\nabla^2 b||b|dx\triangleq\sum_{i=1}^6I_i.
\end{align}
Applying H\"older's inequality and Sobolev's inequality, we have that
\begin{align*}
I_3&\le C\|\sqrt{\rho}\theta\|_{L^2}\|\theta\|_{L^6}\||u|^2\|_{L^6}\|\rho\|_{L^9}^\frac{3}{2}\nonumber\\
&\le C\|\rho\|_{L^\infty}\|\sqrt{\rho}\theta\|_{L^2}\|\nabla\theta\|_{L^2}\||u||\nabla u|\|_{L^2}\|\rho\|_{L^3}^\frac{1}{2},\\
I_4&\le C\|u\|_{L^6}\||b|^2\|_{L^3}\|\nabla E\|_{L^2}+C\|\nabla u\|_{L^2}\||b|^2\|_{L^3}\|E\|_{L^6}\nonumber\\
&\le C\|\nabla u\|_{L^2}\|b\|_{L^6}^2\|\nabla E\|_{L^2}\nonumber\\
&\le \frac18\|\nabla\theta\|_{L^2}^2+C\||u||\nabla u|\|_{L^2}^2
+C\|\nabla b\|_{L^2}^4\|\nabla u\|_{L^2}^2,\\
I_5+I_6&\le C\|E\|_{L^6}\|\nabla^2 b\|_{L^2}\|b\|_{L^3}+
C\|\nabla E\|_{L^2}\|\nabla b\|_{L^6}\|b\|_{L^3}\nonumber\\
&\le C\|\nabla E\|_{L^2}\|b\|_{L^3}\|\nabla^2 b\|_{L^2}\nonumber\\
&\le \frac{1}{6}\|\nabla\theta\|_{L^2}^2+C\||u||\nabla u|\|_{L^2}^2+C(\|b\|_{L^2}\|\nabla b\|_{L^2})\|\nabla^2 b\|_{L^2}^2.
\end{align*}
Substituting the above inequalities into \eqref{2.11} yields that
\begin{align}\label{2.12}
\frac{d}{dt}\|\sqrt{\rho}E\|_{L^2}^2+\|\nabla\theta\|_{L^2}^2
&\le c_1\||u||\nabla u|\|_{L^2}^2+C(\|b\|_{L^2}\|\nabla b\|_{L^2})\|\nabla^2 b\|_{L^2}^2
+C\|\nabla b\|_{L^2}^4\|\nabla u\|_{L^2}^2\nonumber\\
&\quad+C\|\sqrt{\rho}\theta\|_{L^2}\|\nabla\theta\|_{L^2}\||u||\nabla u|\|_{L^2}\|\rho\|_{L^3}^\frac{1}{2}.
\end{align}

2. Exactly in the same way as that in \cite[Lemma 3.4]{LZ20}, we find that
\begin{align}\label{2.9}
\frac{d}{dt}\|\rho^\frac{1}{4}u\|_{L^4}^4+\||u||\nabla u|\|_{L^2}^2
&\le C\int\rho^2\theta^2|u|^2dx+C\|b\|_{L^3}^2\|\nabla^2b\|_{L^2}^2+C\|b\|_{L^3}^4\|\nabla u\|_{L^2}^6\nonumber\\
&\le C(\|b\|_{L^2}\|\nabla b\|_{L^2})\|\nabla^2b\|_{L^2}^2
+C(\|b\|_{L^2}^2\|\nabla b\|_{L^2}^2)\|\nabla u\|_{L^2}^6\nonumber\\
&\quad+C\|\sqrt{\rho}\theta\|_{L^2}\|\theta\|_{L^6}\||u|^2\|_{L^6}\|\rho\|_{L^9}^\frac{3}{2}\nonumber\\
&\le C(\|b\|_{L^2}\|\nabla b\|_{L^2})\|\nabla^2b\|_{L^2}^2
+C(\|b\|_{L^2}^2\|\nabla b\|_{L^2}^2)\|\nabla u\|_{L^2}^6\nonumber\\
&\quad+C\|\sqrt{\rho}\theta\|_{L^2}\|\nabla\theta\|_{L^2}\||u||\nabla u|\|_{L^2}\|\rho\|_{L^3}^\frac{1}{2},
\end{align}
due to H{\"o}lder's inequality and Sobolev's inequality.
Adding \eqref{2.12} to \eqref{2.9} multiplied by $2c_1$, one gets that
\begin{align}\label{2.13}
&\frac{d}{dt}\big(\|\sqrt{\rho}E\|_{L^2}^2+2c_1\|\rho^\frac14u\|_{L^4}^4\big)
+c_1\||u||\nabla u|\|_{L^2}^2+\|\nabla\theta\|_{L^2}^2\nonumber\\
&\le C(\|b\|_{L^2}\|\nabla b\|_{L^2})\|\nabla^2 b\|_{L^2}^2
+C(\|b\|_{L^2}^2\|\nabla b\|_{L^2}^2)\|\nabla u\|_{L^2}^6+C\|\nabla b\|_{L^2}^4\|\nabla u\|_{L^2}^2\nonumber\\
&\quad+C\|\sqrt{\rho}\theta\|_{L^2}\|\nabla\theta\|_{L^2}\||u||\nabla u|\|_{L^2}\|\rho\|_{L^3}^\frac{1}{2}.
\end{align}
From which, the conclusion follows by integrating \eqref{2.13} over $[0,T]$.
\end{proof}

\begin{lemma}
Assume that $3\mu>\lambda$, it holds that
\begin{align}\label{2.14}
\sup_{0\le t\le T}\|\rho\|_{L^3}^3+\int_0^T\int\rho^3pdxdt
&\le \|\rho_0\|_{L^3}^3
+C\sup_{0\le t\le T}\Big(\|\rho\|_{L^\infty}^\frac23\|\sqrt{\rho}u\|_{L^2}^\frac12
\|\sqrt{\rho}|u|^2\|_{L^2}^\frac13\|\rho\|_{L^3}^3\Big)\nonumber\\
&\quad +C\int_0^T\|\rho\|_{L^\infty}^2\|\rho\|_{L^3}^2\|\nabla u\|_{L^2}^2dt
+C\int_0^T\|\rho\|_{L^\infty}\|\rho\|_{L^3}^2\|\nabla b\|_{L^2}^2dt.
\end{align}
\end{lemma}
\begin{proof}
Applying the operator $\Delta^{-1}{\rm div}$ to $\eqref{a1}_2$ leads to
\begin{align}\label{2.15}
\Delta^{-1}\divv(\rho u)_t+\Delta^{-1}\divv\divv(\rho u\otimes u)-(2\mu+\lambda)\divv u+p+\frac{|b|^2}{2}=
\Delta^{-1}\divv\divv(b\otimes b).
\end{align}
In view of $\eqref{a1}_1$, one obtains that
\begin{align}\label{2.16}
\partial_t\rho^3+\divv(u\rho^3)+2\divv u\rho^3=0.
\end{align}
Then, multiplying \eqref{2.15} by $\rho^3$ and using \eqref{2.16}, we have
\begin{align}\label{2.17}
\frac{2\mu+\lambda}{2}(\partial_t\rho^3+\divv(u\rho^3))+\rho^3p
+\rho^3\Delta^{-1}\divv(\rho u)_t+\rho^3\Delta^{-1}\divv\divv(\rho u\otimes u)
=\rho^3\Delta^{-1}\divv\divv(b\otimes b).
\end{align}
Using \eqref{2.16} again, we deduce that
\begin{align*}
\int\rho^3\Delta^{-1}\divv(\rho u)_tdx
&=\frac{d}{dt}\int\rho^3\Delta^{-1}\divv(\rho u)dx
+\int[\divv(\rho^3u)+2\divv u\rho^3]\Delta^{-1}\divv(\rho u)dx\nonumber\\
&\le \int[2\divv u\rho^3\Delta^{-1}\divv(\rho u)-\rho^3u\cdot\nabla\Delta^{-1}\divv(\rho u)]dx
+\frac{d}{dt}\int\rho^3\Delta^{-1}\divv(\rho u)dx,
\end{align*}
which combined with \eqref{2.17} yields that
\begin{align}\label{2.19}
&\frac{d}{dt}\int\Big(\frac{2\mu+\lambda}{2}+\Delta^{-1}\divv (\rho u)\Big)\rho^3dx+\int\rho^3pdx\nonumber\\
&=\int[\rho^3(u\cdot\nabla\Delta^{-1}\divv (\rho u))-\Delta^{-1}\divv \divv (\rho u\otimes u)
-2\divv u\rho^3\Delta^{-1}(\rho u)]dx\nonumber\\
&\quad+\int\rho^3\Delta^{-1}\divv\divv(b\otimes b)dx
\triangleq J_1+J_2.
\end{align}
By virtue of H\"older's inequality and Sobolev's inequality, we derive that
\begin{align*}
|J_1| & \le C\|\rho\|_{L^\infty}^2\|\rho\|_{L^3}^2\|\nabla u\|_{L^2}^2,\\
|J_2| & \le \|\rho\|_{L^\infty}\|\rho\|_{L^3}^2\|\Delta^{-1}\divv \divv (b\otimes b)\|_{L^3}
\le C\|\rho\|_{L^\infty}\|\rho\|_{L^3}^2\|b\|_{L^6}^2\le C\|\rho\|_{L^\infty}\|\rho\|_{L^3}^2\|\nabla b\|_{L^2}^2,
\end{align*}
which together with \eqref{2.19} implies that
\begin{align}\label{2.22}
&\frac{d}{dt}\int\Big(\frac{2\mu+\lambda}{2}+\Delta^{-1}\divv (\rho u)\Big)\rho^3dx+\int\rho^3pdx\nonumber\\
&\leq C\|\rho\|_{L^\infty}^2\|\rho\|_{L^3}^2\|\nabla u\|_{L^2}^2
+C\|\rho\|_{L^\infty}^2\|\rho\|_{L^3}^2\|\nabla b\|_{L^2}^2.
\end{align}
Noticing that
\begin{align*}
\|\Delta^{-1}\divv (\rho u)\|_{L^\infty}&\le C\|\Delta^{-1}\divv (\rho u)\|_{L^6}^\frac13
\|\nabla\Delta^{-1}\divv (\rho u)\|_{L^4}^\frac23\nonumber\\
&\le C\|\rho u\|_{L^2}^\frac12\|\rho u\|_{L^4}^\frac23\le C\|\rho\|_{L^\infty}^\frac23
\|\sqrt{\rho}u\|_{L^2}^\frac13\|\sqrt{\rho}|u|^2\|_{L^2}^\frac13.
\end{align*}
This along with \eqref{2.22} integrated in $t$ over $[0,T]$ leads to \eqref{2.14}.
\end{proof}

\begin{lemma}\label{lem}
Assume that
\begin{align}\label{2.23}
\sup_{0\le t\le T}\|\rho\|_{L^\infty}\le 4\bar{\rho},
\end{align}
then it holds that
\begin{align}\label{2.24}
&\sup_{0\le t\le T}\|\nabla u\|_{L^2}^2+\int_0^T\Big\|\Big(\sqrt{\rho}u_t, \frac{\nabla F}{\sqrt{\bar{\rho}}},
\frac{\nabla w}{\sqrt{\bar{\rho}}}\Big)\Big\|_{L^2}^2dt\nonumber\\
&\le C\|\nabla u_0\|_{L^2}^2+C\bar{\rho}\sup_{0\le t\le T}\|\sqrt{\rho}\theta\|_{L^2}^2+C\sup_{0\le t\le T}\|b\|_{L^4}^4
+\eta\bar{\rho}\int_0^T\||b||\nabla b|\|_{L^2}^2dt
\nonumber\\
&\quad+C\bar{\rho}^3\int_0^T\|\nabla u\|_{L^2}^4\big(\|\nabla u\|_{L^2}^2+\bar{\rho}\|\sqrt{\rho}\theta\|_{L^2}^2+\|b\|_{L^4}^4\big)dt
+C\bar{\rho}\int_0^T\|\nabla b\|_{L^2}^4\|\nabla u\|_{L^2}^2dt\nonumber\\
&\quad+C\int_0^T\Big(\bar{\rho}^2\|\rho\|_{L^3}^\frac12
\|\sqrt{\rho}\theta\|_{L^2}+\bar{\rho}\Big)
\big(\|\nabla\theta\|_{L^2}^2+\||u||\nabla u|\|_{L^2}^2\big)dt
+\varepsilon_1\bar{\rho}\int_0^T\|\nabla\theta\|_{L^2}^2dt\nonumber\\
&\quad+C\bar{\rho}\int_0^T\|b\|_{L^2}\|\nabla b\|_{L^2}\|b_t\|_{L^2}^2dt,
\end{align}
where $F=(2\mu+\lambda)\divv u-p-\frac{|b|^2}{2}$ and $w=\curl u$.
\end{lemma}
\begin{proof}[Proof]
Multiplying $\eqref{a1}_2$ by $u_t$ and integration by parts, we get that
\begin{align}\label{2.25}
&\frac12\frac{d}{dt}(\mu\|\nabla u\|_{L^2}^2+(\mu+\lambda)\|\divv u\|_{L^2}^2)-\int p\divv u_tdx
+\|\sqrt{\rho}u_t\|_{L^2}^2\nonumber\\
&=-\int\rho u\cdot\nabla u\cdot u_tdx-\int b\cdot\nabla b\cdot u_tdx-\frac12\int\nabla|b|^2\cdot u_tdx.
\end{align}
By the definition of effective viscous flux $F$, we use $\divv u=\frac{F+p+\frac12|b|^2}{2\mu+\lambda}$ to obtain that
\begin{align}
\int p\divv u_tdx&=\frac{d}{dt}\int p\divv udx-\int p_t\divv udx\nonumber\\
&=\frac{d}{dt}\int p{\rm div}udx-\frac{1}{2(2\mu+\lambda)}\frac{d}{dt}\|p\|_{L^2}^2-\frac{1}{2\mu+\lambda}\int p_tFdx-\frac{1}{2(2\mu+\lambda)}\int p_t|b|^2dx\nonumber\\
&=\frac{1}{2(2\mu+\lambda)}\frac{d}{dt}\|p\|_{L^2}^2+\frac{1}{2(2\mu+\lambda)}\frac{d}{dt}\int p|b|^2dx
+\frac{1}{2\mu+\lambda}\frac{d}{dt}\int pFdx\nonumber\\
&\quad-\frac{1}{2\mu+\lambda}\int p_t\Big(F+\frac12|b|^2\Big)dx.
\end{align}
It follows $\eqref{a1}_1$, $\eqref{a1}_3$, and the state equation that
\begin{align}
p_t&=-\divv (pu)-(\gamma-1)\big(p\divv u-\Delta\theta-\mathcal{Q}(\nabla u)
-|\curl b|^2\big)
\end{align}
with $\gamma-1=\frac{R}{c_v}$,
which leads to
\begin{align}\label{2.28}
&\int p_t\Big(F+\frac12|b|^2\Big)dx \notag \\
& =\int\big[(\gamma-1)\big(\mathcal{Q}(\nabla u)-p\divv u+|\curl b|^2\big)F
+(up-(\gamma-1)\nabla\theta)\cdot\nabla F\big]dx\nonumber\\
& \quad+\frac12\int\big[(\gamma-1)\big(\mathcal{Q}(\nabla u)-p\divv u+|\curl b|^2\big)|b|^2
+(up-(\gamma-1)\nabla\theta)\cdot\nabla |b|^2\big]dx.
\end{align}
In view of $\|\nabla u\|_{L^2}^2=\|w\|_{L^2}^2+\|\divv u\|_{L^2}^2$,
then one obtains from \eqref{2.25}--\eqref{2.28} that
\begin{align}\label{2.29}
&\frac12\frac{d}{dt}\Big(\mu\|w\|_{L^2}^2
+\frac{\|F\|_{L^2}^2}{2\mu+\lambda}+\frac{1}{2(2\mu+\lambda)}\int|b|^2Fdx+\frac{1}{2(2\mu+\lambda)}\int|b|^4dx\Big)+\|\sqrt{\rho}u_t\|_{L^2}^2\nonumber\\
&=-\int\rho u\cdot\nabla u\cdot u_tdx
+\int b\cdot\nabla b\cdot u_tdx-\frac12\int\nabla|b|^2\cdot u_tdx\nonumber\\
&\quad-\frac{\gamma-1}{2\mu+\lambda}\int\Big[\mathcal{Q}(\nabla u)
-p\divv u+|\curl b|^2\Big]\Big(F+\frac12|b|^2\Big)dx\nonumber\\
&\quad+\frac{1}{2\mu+\lambda}\int((\gamma-1)\nabla\theta-up)\cdot\nabla\Big(F+\frac12|b|^2\Big)dx.
\end{align}
Using $\Delta u=\nabla\divv u-\curl w$ to rewrite $\eqref{a1}_2$ as follows
\begin{align}\label{2.30}
\rho u_t+\rho u\cdot\nabla u=\nabla F-\mu\curl w+b\cdot\nabla b.
\end{align}
Multiplying \eqref{2.30} by $\nabla F$ and using $\int\nabla F\cdot\curl wdx=0$ and \eqref{2.23}, we have
\begin{align*}
\|\nabla F\|_{L^2}^2
&=\int[\rho(u_t+u\cdot\nabla u)-b\cdot\nabla b)\cdot\nabla F]dx\nonumber\\
&\le \int\Big(\frac{|\nabla F|^2}{2}+2\bar{\rho}\rho|u_t|^2\Big)dx
+\int\rho u\cdot\nabla u\cdot\nabla Fdx-\int b\cdot\nabla b\cdot\nabla Fdx,
\end{align*}
which yields that
\begin{align}\label{2.31}
\frac{\|\nabla F\|_{L^2}^2}{16\bar{\rho}}&\le
\frac{1}{4}\|\sqrt{\rho}u_t\|_{L^2}^2
+\frac{1}{8\bar{\rho}}\int\rho u\cdot\nabla u\cdot\nabla Fdx
-\frac{1}{8\bar{\rho}}\int b\cdot\nabla b\cdot\nabla Fdx.
\end{align}
Similarly, one deduces that
\begin{align}\label{2.32}
\frac{\mu\|\nabla w\|_{L^2}^2}{16\bar{\rho}}
&\le
\frac{1}{4}\|\sqrt{\rho}u_t\|_{L^2}^2
+\frac{1}{8\bar{\rho}}\int\rho u\cdot\nabla u\cdot\curl wdx
-\frac{1}{8\bar{\rho}}\int b\cdot\nabla b\cdot\curl wdx.
\end{align}
Putting \eqref{2.31} and \eqref{2.32} into \eqref{2.29} gives that
\begin{align}\label{2.33}
&\frac12\frac{d}{dt}\Big(\mu\|w\|_{L^2}^2
+\frac{\|F\|_{L^2}^2}{2\mu+\lambda}+\frac{1}{2\mu+\lambda}\int|b|^2Fdx
+\frac{1}{2\mu+\lambda}\int|b|^4dx\Big)\nonumber\\
&\quad+\frac12\|\sqrt{\rho}u_t\|_{L^2}^2+\frac{1}{16\bar{\rho}}(\|\nabla F\|_{L^2}^2
+\mu^2\|\nabla w\|_{L^2}^2)\nonumber\\
&\le C\int\rho|u||\nabla u|\Big[|u_t|+\frac{1}{\bar{\rho}}(|\nabla F|+|\nabla w|)\Big]dx
+C\int(|\nabla\theta|+\rho\theta|u|)(|\nabla F|+|b||\nabla b|)dx\nonumber\\
&\quad+C\int(|\nabla u|^2+\rho\theta|\nabla u|)(|F|+\frac12|b|^2)+\frac{C}{\bar{\rho}}\int|b||\nabla b|(|\nabla F|+|\nabla w|)dx\nonumber\\
&\quad+C\int|\nabla b|^2Fdx+\int b\cdot\nabla b\cdot u_tdx-\frac12\int\nabla|b|^2\cdot u_tdx
\triangleq\sum_{i=1}^7Q_i.
\end{align}
By \eqref{2.23}, H{\"o}lder's inequality, and Young's inequality, we get that
\begin{align*}
Q_1 & \le C\sqrt{\bar{\rho}}\||u||\nabla u|\|_{L^2}\|\sqrt{\rho}u_t\|_{L^2}
+C\||u||\nabla u|\|_{L^2}(\|\nabla F\|_{L^2}+\|\nabla w\|_{L^2})\nonumber\\
& \le \frac14\|\sqrt{\rho}u_t\|_{L^2}^2+\frac{1}{224\bar{\rho}}(\|\nabla F\|_{L^2}^2+\mu^2\|\nabla w\|_{L^2}^2)
+C\bar{\rho}\||u||\nabla u|\|_{L^2}^2,\\
Q_2 & \le C\|\rho\theta u\|_{L^2}
(\|\nabla F\|_{L^2}+\||b||\nabla b|\|_{L^2})
+C\|\nabla\theta\|_{L^2}(\|\nabla F\|_{L^2}+\||b||\nabla b|\|_{L^2})\nonumber\\
&\le C\sqrt{\bar{\rho}}\|\rho\|_{L^3}^\frac14\|\sqrt{\rho}\theta\|_{L^2}^\frac12
\|\nabla\theta\|_{L^2}^\frac12\||u||\nabla u|\|_{L^2}^\frac12(\|\nabla F\|_{L^2}+\||b||\nabla b|\|_{L^2})\nonumber\\
&\quad+C\|\nabla\theta\|_{L^2}(\|\nabla F\|_{L^2}+\||b||\nabla b|\|_{L^2}^2)\nonumber\\
&\le \frac{1}{224\bar{\rho}}\|\nabla F\|_{L^2}^2+\varepsilon_1\bar{\rho}\||b||\nabla b|\|_{L^2}^2+C(\bar{\rho}^2\|\rho\|_{L^3}^\frac12\|\sqrt{\rho}\theta\|_{L^2}+\bar{\rho})
(\|\nabla\theta\|_{L^2}^2+\||u||\nabla u|\|_{L^2}^2).
\end{align*}
Noticing that
\begin{align}\label{3.30}
\|\nabla u\|_{L^6}
&\le C(\|w\|_{L^6}+\|\divv u\|_{L^6})\nonumber\\
&\le C\big(\|w\|_{L^6}+\|F\|_{L^6}+\|\rho\theta\|_{L^6}
+\||b|^2\|_{L^6}\big)\nonumber\\
&\le C\big(\|\nabla w\|_{L^2}+\|\nabla F\|_{L^2}+\bar{\rho}\|\nabla\theta\|_{L^2}+\||b||\nabla b|\|_{L^2}\big).
\end{align}
Then it follows from H\"older's, Young's, and Gagliardo-Nirenberg inequalities that
\begin{align*}
Q_3&\le C\|\nabla u\|_{L^2}\|\nabla u\|_{L^6}(\|F\|_{L^3}+\|b\|_{L^6}^2)
+C\|\nabla u\|_{L^2}\|\rho\theta\|_{L^6}(\|F\|_{L^3}+\|b\|_{L^6}^2)\nonumber\\
&\le C\|\nabla u\|_{L^2}(\|\nabla w\|_{L^2}+\|\nabla F\|_{L^2}+\bar{\rho}\|\nabla\theta\|_{L^2}+\||b||\nabla b|\|_{L^2})
\Big(\|F\|_{L^2}^\frac12\|\nabla F\|_{L^2}^\frac12
+\|\nabla b\|_{L^2}^2\Big)\nonumber\\
&\quad+C\bar{\rho}\|\nabla u\|_{L^2}\|\nabla\theta\|_{L^2}\Big(\|F\|_{L^2}^\frac12\|\nabla F\|_{L^2}^\frac12
+\|\nabla b\|_{L^2}^2\Big)\nonumber\\
&\le \frac{1}{224\bar{\rho}}(\|\nabla F\|_{L^2}^2+\mu^2\|\nabla w\|_{L^2}^2)
+C\bar{\rho}^3\|\nabla u\|_{L^2}^4\|F\|_{L^2}^2+C\bar{\rho}\|\nabla\theta\|_{L^2}^2\nonumber\\
&\quad+C\bar{\rho}\|\nabla b\|_{L^2}^4\|\nabla u\|_{L^2}^2+\varepsilon_1\||b||\nabla b|\|_{L^2}^2.
\end{align*}
By virtue of H\"older's, Young's, and Gagliardo-Nirenberg inequalities, one obtains that
\begin{align*}
Q_4&\le C{\bar{\rho}}\|b\|_{L^3}\|\nabla b\|_{L^6}(\|\nabla F\|_{L^2}+\|\nabla w\|_{L^2})\nonumber\\
&\le \frac{1}{224\bar{\rho}}(\|\nabla F\|_{L^2}^2+\mu^2\|\nabla w\|_{L^2}^2)+C\|b\|_{L^2}\|\nabla b\|_{L^2}\|\nabla^2b\|_{L^2}^2,\\
Q_5&\le C\int(|b||\nabla^2b||F|+|b||\nabla b||\nabla F|)dx
\le C\|b\|_{L^3}\|\nabla^2b\|_{L^2}\|F\|_{L^6}\nonumber\\
&\le \frac{1}{224\bar{\rho}}\|\nabla F\|_{L^2}^2+C\bar{\rho}\|b\|_{L^2}\|\nabla b\|_{L^2}\|\nabla^2b\|_{L^2}^2.
\end{align*}
Using H\"older's inequality and \eqref{3.30}, we arrive at
\begin{align*}
Q_6&=-\frac{d}{dt}\int b\cdot\nabla u\cdot bdx+\int b_t\cdot\nabla u\cdot bdx+\int b\cdot\nabla u\cdot b_tdx\nonumber\\
&\le -\frac{d}{dt}\int b\cdot\nabla u\cdot bdx+C\|b\|_{L^3}\|b_t\|_{L^2}\|\nabla u\|_{L^6}\nonumber\\
&\le -\frac{d}{dt}\int b\cdot\nabla u\cdot bdx+C\bar{\rho}\|b\|_{L^2}\|\nabla b\|_{L^2}\|b_t\|_{L^2}^2
+\frac{\varepsilon_1}{4\bar{\rho}}\|\nabla u\|_{L^6}^2\nonumber\\
&\le  -\frac{d}{dt}\int b\cdot\nabla u\cdot bdx+\frac{1}{224\bar{\rho}}(\|\nabla F\|_{L^2}^2+\mu^2\|\nabla w\|_{L^2}^2)
+\frac{\varepsilon_1\bar{\rho}}{4}\|\nabla\theta\|_{L^2}^2\nonumber\\
&\quad+\frac{\varepsilon_1 C}{2\bar{\rho}}\||b||\nabla b|\|_{L^2}^2+C\bar{\rho}\|b\|_{L^2}\|\nabla b\|_{L^2}\|b_t\|_{L^2}^2.
\end{align*}
Similarly, we get that
\begin{align*}
Q_7&\le -\frac{1}{2}\frac{d}{dt}\int|b|^2\divv udx
+\frac{1}{224\bar{\rho}}\big(\|\nabla F\|_{L^2}^2+\mu^2\|\nabla w\|_{L^2}^2\big)
+\frac{\varepsilon_1\bar{\rho}}{4}\|\nabla\theta\|_{L^2}^2\nonumber\\
&\quad+\frac{\varepsilon_1 C}{2\bar{\rho}}\||b||\nabla b|\|_{L^2}^2
+C\bar{\rho}\|b\|_{L^2}\|\nabla b\|_{L^2}\|b_t\|_{L^2}^2.
\end{align*}
Substituting all those estimates for $Q_i\ (i=1, 2,\cdots,7)$ into \eqref{2.33} leads to
\begin{align*}
&\frac{d}{dt}\Big(\mu\|w\|_{L^2}^2
+\frac{\|F\|_{L^2}^2}{2\mu+\lambda}+\frac{1}{2\mu+\lambda}\int|b|^2Fdx
+\frac{1}{2\mu+\lambda}\int|b|^4dx\Big)\nonumber\\
&\quad+\frac12\|\sqrt{\rho}u_t\|_{L^2}^2+\frac{1}{16\bar{\rho}}(\|\nabla F\|_{L^2}^2
+\mu^2\|\nabla w\|_{L^2}^2)\nonumber\\
&\le -\frac{d}{dt}\int \big(b\cdot\nabla u\cdot b+|b|^2\divv u\big)dx+\varepsilon_1\bar{\rho}\||b||\nabla b|\|_{L^2}^2
+C\bar{\rho}^3\|\nabla u\|_{L^2}^4\|F\|_{L^2}^2
\nonumber\\
&\quad+C\big(\bar{\rho}^2\|\rho\|_{L^3}^\frac12\|\sqrt{\rho}\theta\|_{L^2}
+\bar{\rho}\big)
\big(\|\nabla\theta\|_{L^2}^2+\||u||\nabla u|\|_{L^2}^2\big)+C\bar{\rho}\|\nabla b\|_{L^2}^4\|\nabla u\|_{L^2}^2\nonumber\\
&\quad+C\bar{\rho}\|b\|_{L^2}\|\nabla b\|_{L^2}\|b_t\|_{L^2}^2+\varepsilon_1\bar{\rho}\|\nabla\theta\|_{L^2}^2,
\end{align*}
which integrating in $t$ over $[0,T]$, together with the following facts
\begin{align*}
&\|\nabla u\|_{L^2}\le C\big(\|w\|_{L^2}+\|F\|_{L^2}+\|\rho\theta\|_{L^2}+\|b\|_{L^4}^2\big)
\le C\big(\|w\|_{L^2}+\|F\|_{L^2}+\sqrt{\bar{\rho}}\|\sqrt{\rho}\theta\|_{L^2}
+\|b\|_{L^4}^2\big),\\
&\Big|-\int \big(b\cdot\nabla u\cdot b+|b|^2\divv u\big)dx\Big|\le
C\|\nabla u\|_{L^2}\||b|^2\|_{L^2}\le \varepsilon_2\|\nabla u\|_{L^2}^2+C\|b\|_{L^4}^4,\\
& \int|b|^2Fdx\le \varepsilon_3\|F\|_{L^2}^2+C\|b\|_{L^4}^4,
\end{align*}
with sufficiently small constants $\varepsilon_2$ and $\varepsilon_3$, yields \eqref{2.24}.
\end{proof}

\begin{lemma}\label{l27}
Let \eqref{2.23} be satisfied, then it holds that
\begin{align*}
&\sup_{0\le t\le T}\|\rho\|_{L^\infty}\nonumber\\
&\le
\|\rho_0\|_{L^\infty}e^{C\bar{\rho}^\frac23
\sup\limits_{0\le t\le T}
\big(\|\sqrt{\rho}u\|_{L^2}^\frac13\|\sqrt{\rho}|u|^2\|_{L^2}^\frac13\big)
+C\bar{\rho}\int_0^T\|\nabla u\|_{L^2}\|(\nabla F, \nabla w, \bar{\rho}\nabla\theta)\|_{L^2}dt
+C\big(\int_0^T\|\nabla b\|_{L^2}^2dt\big)^\frac12\big(\int_0^T\|\nabla^2b\|_{L^2}^2dt\big)^\frac12}.
\end{align*}
\end{lemma}
\begin{proof}[Proof]
According to \eqref{2.15}, one obtains that
\begin{align*}
&\Delta^{-1}\divv(\rho u)_t+u\cdot\nabla\Delta^{-1}\divv (\rho u)-(2\mu+\lambda)\divv u+p+\frac{|b|^2}{2}\Delta^{-1}\divv\divv  (b\otimes b)\nonumber\\
&=u\cdot\nabla\Delta^{-1}\divv (\rho u)-\Delta^{-1}\divv \divv (\rho u\otimes u).
\end{align*}
Similarly to \cite[Proposition 2.6]{L20}, we can deduce that
\begin{align}\label{2.42}
\sup_{0\le t\le T}\|\rho\|_{L^\infty}\le\|\rho_0\|_{L^\infty}e^{C\bar{\rho}^\frac23
\sup\limits_{0\le t\le T}
\big(\|\sqrt{\rho}u\|_{L^2}^\frac13\|\sqrt{\rho}|u|^2\|_{L^2}^\frac13\big)
+C\bar{\rho}\int_0^T\|\nabla u\|_{L^2}\|(\nabla F, \nabla w, \bar{\rho}\nabla\theta)\|_{L^2}dt
+C\int_0^T\|b\|_{L^\infty}^2dt}.
\end{align}
Employing Gagliardo-Nirenberg and H\"older's inequalities, one gets that
\begin{align*}
\int_0^T\|b\|_{L^\infty}^2dt\le C\int_0^T\|\nabla b\|_{L^2}\|\nabla^2b\|_{L^2}dt\le C\Big(\int_0^T\|\nabla b\|_{L^2}^2dt\Big)^\frac12\Big(\int_0^T\|\nabla^2b\|_{L^2}^2dt\Big)^\frac12.
\end{align*}
This together with \eqref{2.42} implies the conclusion immediately.
\end{proof}

\begin{lemma}\label{l28}
Let
\begin{align*}
N_T\triangleq\bar{\rho}\big[\|\rho\|_{L^3}+\bar{\rho}^2(\|\sqrt{\rho}u\|_{L^2}^2
+\|b\|_{L^2}^2)\big]\big[\|\nabla u\|_{L^2}^2+
\bar{\rho}(\|\sqrt{\rho}E\|_{L^2}^2+\|\nabla b\|_{L^2}^2)\big].
\end{align*}
There exists a positive constant $\eta_0$ depending only on $R$, $\mu$, and $\lambda$ such that if
\begin{align}\label{2.44}
\sup_{0\le t\le T}\|\rho\|_{L^\infty}\le 4\bar{\rho}, \ N_T\le \sqrt{\eta},\ \eta\le\eta_0,
\end{align}
then it holds that
\begin{align}
&\sup_{0\le t\le T}\|\rho\|_{L^3}+\Big(\int_0^T\int\rho^3pdxdt\Big)^\frac13\le C\big[\|\rho_0\|_{L^3}
+\bar{\rho}^2(\|\sqrt{\rho_0}u_0\|_{L^2}^2+\|b_0\|_{L^2}^2)\big],\label{2.45}\\
&\bar{\rho}^2\Big(\sup_{0\le t\le T}(\|\sqrt{\rho}u\|_{L^2}^2+\|b\|_{L^2}^2)
+\int_0^T\|(\nabla u, \nabla b)\|_{L^2}^2dt\Big)\nonumber\\
&\le C(\|\rho_0\|_{L^3}
+\bar{\rho}^2(\|\sqrt{\rho_0}u_0\|_{L^2}^2+\|b_0\|_{L^2}^2)),\label{2.46}\\[3pt]
&\sup_{0\le t\le T}(\bar{\rho}(\|\nabla b\|_{L^2}^2+\|b\|_{L^4}^4+\|\sqrt{\rho}E\|_{L^2}^2)+\|\nabla u\|_{L^2}^2)\nonumber\\
&\quad +\bar{\rho}\int_0^T\Big(\|(b_t, \nabla^2b, |b||\nabla b|, \nabla\theta, |u||\nabla u|)\|_{L^2}^2
+\Big\|\Big(\sqrt{\rho}u_t, \frac{\nabla F}{\sqrt{\bar{\rho}}}, \frac{\nabla w}{\sqrt{\bar{\rho}}}\Big)\Big\|_{L^2}^2\Big)dt\nonumber\\
&\le C(\bar{\rho}\|\nabla b_0\|_{L^2}^2+\|\sqrt{\rho_0}E_0\|_{L^2}^2+\|\nabla u_0\|_{L^2}^2),\label{2.47}\\
&\sup_{0\le t\le T}\|\rho\|_{L^\infty}\le \bar{\rho}e^{CN_0^\frac16+CN_0^\frac12}. \label{2.48}
\end{align}
\end{lemma}
\begin{proof}[Proof]
1. By virtue of \eqref{2.2}, \eqref{2.44}, and $\bar{\rho}=\|\rho_0\|_{L^\infty}+1$, we deduce that
\begin{align}\label{2.49}
\bar{\rho}\int_0^T\|\nabla u\|_{L^2}^6&\le C\bar{\rho}\sup_{0\le t\le T}\|\nabla u\|_{L^2}^4\int_0^T\|\nabla u\|_{L^2}^2dt\nonumber\\
&\le C\bar{\rho}\sup_{0\le t\le T}\|\nabla u\|_{L^2}^4\Big(\|\sqrt{\rho_0}u_0\|_{L^2}^2+\|b_0\|_{L^2}^2
+\sup_{0\le t\le T}\|\rho\|_{L^3}^2\int_0^T\|\nabla\theta\|_{L^2}^2dt\Big)\nonumber\\
&\le C\bar{\rho}\sup_{0\le t\le T}\|\nabla u\|_{L^2}^4\Big(\sup_{0\le t\le T}(\|\sqrt{\rho}u\|_{L^2}^2+\|b\|_{L^2}^2)
+\sup_{0\le t\le T}\|\rho\|_{L^3}^2\int_0^T\|\nabla\theta\|_{L^2}^2dt\Big)\nonumber\\
&\le C\eta^\frac12\sup_{0\le t\le T}\|\nabla u\|_{L^2}^2
+C\eta\int_0^T\|\nabla\theta\|_{L^2}^2dt.
\end{align}
It follows from \eqref{2.4} and \eqref{2.44} that
\begin{align}\label{2.50}
&\bar{\rho}
\sup_{0\le t\le T}(\|\nabla b\|_{L^2}^2+\|b\|_{L^4}^4)+\bar{\rho}\int_0^T(\|b_t\|_{L^2}^2+\|\nabla^2b\|_{L^2}^2
+\||b||\nabla b|\|_{L^2}^2)dt\nonumber\\
&\le C\bar{\rho}(\|\nabla b_0\|_{L^2}^2+\|b_0\|_{L^4}^4)+C\eta^\frac12\int_0^T\|\nabla u\|_{L^2}^6dt
+C\eta^\frac34\int_0^T\|\nabla^2b\|_{L^2}^2dt,
\end{align}
which, after choosing $\eta_0$ suitably small, together with \eqref{2.49} yields that
\begin{align}\label{2.51}
&\bar{\rho}
\sup_{0\le t\le T}(\|\nabla b\|_{L^2}^2+\|b\|_{L^4}^4)+\bar{\rho}\int_0^T(\|b_t\|_{L^2}^2+\|\nabla^2b\|_{L^2}^2
+\||b||\nabla b|\|_{L^2}^2)dt\nonumber\\
&\le C\bar{\rho}(\|\nabla b_0\|_{L^2}^2+\|b_0\|_{L^4}^4)+C\eta\sup_{0\le t\le T}\|\nabla u\|_{L^2}^2
+C\eta^\frac32\int_0^T\|\nabla\theta\|_{L^2}^2dt.
\end{align}

2. We infer from Gagliardo-Nirenberg and H\"older's inequalities,  \eqref{2.8}, \eqref{2.44}, and \eqref{2.50} that
\begin{align*}
&\sup_{0\le t\le T}\bar{\rho}\|\sqrt{\rho}E\|_{L^2}^2
+\bar{\rho}\int_0^T\big(\||u||\nabla u|\|_{L^2}^2+\|\nabla\theta\|_{L^2}^2\big)dt\nonumber\\
&\le C\bar{\rho}\|\sqrt{\rho_0}E_0\|_{L^2}^2
+C\eta^\frac14\bar{\rho}\int_0^T\|\nabla^2 b\|_{L^2}^2dt\nonumber\\
&\quad+C\eta\bar{\rho}\int_0^T\|\nabla u\|_{L^2}^6dt+C\bar{\rho}\sup_{0\le t\le T}\|\nabla b\|_{L^2}^4\int_0^T\|\nabla u\|_{L^2}^2dt\nonumber\\
&\quad+C\bar{\rho}\int_0^T\|\sqrt{\rho}\theta\|_{L^2}\|\nabla\theta\|_{L^2}\||u||\nabla u|\|_{L^2}\|\rho\|_{L^3}^\frac{1}{2}dt\nonumber\\
&\le C\bar{\rho}(\|\sqrt{\rho_0}E_0\|_{L^2}^2+\|\nabla b_0\|_{L^2}^2
+(\|b_0\|_{L^2}\|\nabla b_0\|_{L^2})\|\nabla b_0\|_{L^2}^2)+C\eta^\frac14\bar{\rho}\int_0^T\|\nabla^2b\|_{L^2}^2dt\nonumber\\
&\quad
+C\bar{\rho}\sup_{0\le t\le T}\|\nabla b\|_{L^2}^4(\sup_{0\le t\le T}(\|\sqrt{\rho}u\|_{L^2}^2+\|b\|_{L^2}^2)+\sup_{0\le t\le T}\|\rho\|_{L^3}^2\int_0^T\|\nabla\theta\|_{L^2}^2dt)\nonumber\\
&\quad+C\eta^\frac32\sup_{0\le t\le T}\|\nabla u\|_{L^2}^2
+C\eta\bar{\rho}\int_0^T\|\nabla\theta\|_{L^2}^2dt
+C\eta^\frac14\bar{\rho}\int_0^T\|(|\nabla\theta|,|u||\nabla u|)\|_{L^2}^2dt\nonumber\\
&\le C\bar{\rho}(\|\sqrt{\rho_0}E_0\|_{L^2}^2+\|\nabla b_0\|_{L^2}^2)
+C\bar{\rho}\eta^\frac12\sup_{0\le t\le T}\|\nabla b\|_{L^2}^2
+C\eta^\frac32\sup_{0\le t\le T}\|\nabla u\|_{L^2}^2\nonumber\\
&\quad+C\eta^\frac14\bar{\rho}\int_0^T\|\nabla^2b\|_{L^2}^2dt
+C\eta^\frac14\bar{\rho}\int_0^T\|(|\nabla\theta|,|u||\nabla u|)\|_{L^2}^2dt.
\end{align*}
This, combined with the fact $\eta_0$ is suitably small, implies that
\begin{align}\label{2.52}
&\sup_{0\le t\le T}\bar{\rho}\|\sqrt{\rho}E\|_{L^2}^2
+\bar{\rho}\int_0^T(\||u||\nabla u|\|_{L^2}^2+\|\nabla\theta\|_{L^2}^2)dt\nonumber\\
&\le C\bar{\rho}(\|\sqrt{\rho_0}E_0\|_{L^2}^2+\|\nabla b_0\|_{L^2}^2)
+C\bar{\rho}\eta^\frac12\sup_{0\le t\le T}\|\nabla b\|_{L^2}^2
+C\eta^\frac32\sup_{0\le t\le T}\|\nabla u\|_{L^2}^2 \notag \\
& \quad +C\eta^\frac14\bar{\rho}\int_0^T\|\nabla^2b\|_{L^2}^2dt.
\end{align}

3. Using \eqref{2.24}, \eqref{2.51}, and \eqref{2.52}, and choosing $\varepsilon_1\le \eta$, we deduce from \eqref{2.44} that
\begin{align}\label{2.53}
&\sup_{0\le t\le T}\|\nabla u\|_{L^2}^2+\int_0^T\Big\|\Big(\sqrt{\rho}u_t, \frac{\nabla F}{\sqrt{\bar{\rho}}},
\frac{\nabla w}{\sqrt{\bar{\rho}}}\Big)\Big\|_{L^2}^2dt\nonumber\\
&\le C\|\nabla u_0\|_{L^2}^2+C\bar{\rho}(\sup_{0\le t\le T}\|\sqrt{\rho_0}E_0\|_{L^2}^2+\|\nabla b_0\|_{L^2}^2)
+\eta\bar{\rho}\int_0^T\||b||\nabla b|\|_{L^2}^2dt
\nonumber\\
&\quad+C\bar{\rho}^3\int_0^T\|\nabla u\|_{L^2}^4(\|\nabla u\|_{L^2}^2+\bar{\rho}\|\sqrt{\rho}E\|_{L^2}^2+\|b\|_{L^4}^4)dt
+C\bar{\rho}\int_0^T\|\nabla b\|_{L^2}^4\|\nabla u\|_{L^2}^2dt\nonumber\\
&\quad+C\sup_{0\le t\le T}(\bar{\rho}^2\|\rho\|_{L^3}^\frac12\|\sqrt{\rho}\theta\|_{L^2})\int_0^T
(\|\nabla\theta\|_{L^2}^2+\||u||\nabla u|\|_{L^2}^2)dt
+\eta\bar{\rho}\int_0^T\|\nabla\theta\|_{L^2}^2dt\nonumber\\
&\quad+C\bar{\rho}\int_0^T\|b\|_{L^2}\|\nabla b\|_{L^2}\|b_t\|_{L^2}^2dt+C\bar{\rho}\eta^\frac12\sup_{0\le t\le T}\|\nabla b\|_{L^2}^2
+C\eta^\frac32\sup_{0\le t\le T}\|\nabla u\|_{L^2}^2\nonumber\\
&\quad+C\eta^\frac14\bar{\rho}\int_0^T\|\nabla^2b\|_{L^2}^2dt+C\bar{\rho}\eta^\frac12\int_0^T\|\nabla u\|_{L^2}^6dt.
\end{align}
In view of \eqref{2.44} and \eqref{2.2}, we obtain after choosing $\eta_0<1$ that
\begin{align}\label{2.54}
&\bar{\rho}^3\int_0^T\|\nabla u\|_{L^2}^4(\|\nabla u\|_{L^2}^2+\bar{\rho}\|\sqrt{\rho}E\|_{L^2}^2+\|b\|_{L^4}^4)dt\nonumber\\
&\le C\bar{\rho}^3\sup_{0\le t\le T}(\|\nabla u\|_{L^2}^2+\bar{\rho}\|\sqrt{\rho}E\|_{L^2}^2+\eta^\frac14\|\nabla b\|_{L^2}^2)
\sup_{0\le t\le T}\|\nabla u\|_{L^2}^2\int_0^T\|\nabla u\|_{L^2}^2dt\nonumber\\
&\le C\bar{\rho}^3(\|\nabla u\|_{L^2}^2+\bar{\rho}(\|\sqrt{\rho}E\|_{L^2}^2+\|\nabla b\|_{L^2}^2))\sup_{0\le t\le T}\|\nabla u\|_{L^2}^2\nonumber\\
&\quad\times\Big(\sup_{0\le t\le T}(\|\sqrt{\rho}u\|_{L^2}^2+\|b\|_{L^2}^2)
+C\sup_{0\le t\le T}\|\rho\|_{L^3}^2\int_0^T\|\nabla\theta\|_{L^2}^2dt\Big)\nonumber\\
&\le C\eta^\frac12\sup_{0\le t\le T}\|\nabla u\|_{L^2}^2
+C\eta\int_0^T\|\nabla\theta\|_{L^2}^2dt,
\end{align}
and
\begin{align}\label{2.55}
\bar{\rho}\int_0^T\|\nabla b\|_{L^2}^4\|\nabla u\|_{L^2}^2dt
&\le C\bar{\rho}\sup_{0\le t\le T}\|\nabla b\|_{L^2}^4\Big(\sup_{0\le t\le T}(\|\sqrt{\rho}u\|_{L^2}^2
+\|b\|_{L^2}^2)+\|\rho\|_{L^3}^2\int_0^T\|\nabla\theta\|_{L^2}^2dt\Big)\nonumber\\
&\le C\eta^\frac14\sup_{0\le t\le T}\|\nabla b\|_{L^2}^2+C\eta\int_0^T\|\nabla\theta\|_{L^2}^2dt.
\end{align}
Putting \eqref{2.54} and \eqref{2.55} into \eqref{2.53}, we derive that
\begin{align*}
&\sup_{0\le t\le T}\|\nabla u\|_{L^2}^2+\int_0^T\Big\|\Big(\sqrt{\rho}u_t, \frac{\nabla F}{\sqrt{\bar{\rho}}},
\frac{\nabla w}{\sqrt{\bar{\rho}}}\Big)\Big\|_{L^2}^2dt\nonumber\\
&\le C\|\nabla u_0\|_{L^2}^2+C\bar{\rho}\Big(\sup_{0\le t\le T}\|\sqrt{\rho_0}E_0\|_{L^2}^2+\|\nabla b_0\|_{L^2}^2\Big)
+\eta\bar{\rho}\int_0^T\||b||\nabla b|\|_{L^2}^2dt
\nonumber\\
&\quad+C\eta^\frac12\sup_{0\le t\le T}\|\nabla u\|_{L^2}^2
+C\eta\bar{\rho}\int_0^T\|\nabla\theta\|_{L^2}^2dt
+C\eta^\frac14\bar{\rho}\int_0^T\|b_t\|_{L^2}^2dt\nonumber\\
&\quad+C\bar{\rho}\eta^\frac14\sup_{0\le t\le T}\|\nabla b\|_{L^2}^2
+C\eta^\frac32\sup_{0\le t\le T}\|\nabla u\|_{L^2}^2+C\eta^\frac14\bar{\rho}\int_0^T\|\nabla^2b\|_{L^2}^2dt\nonumber\\
&\quad+C\eta^\frac14\bar{\rho}\int_0^T
\big(\|\nabla\theta\|_{L^2}^2+\||u||\nabla u|\|_{L^2}^2\big)dt,
\end{align*}
which, after choosing $\eta_0$ suitably small, together with \eqref{2.51} and \eqref{2.52} yields that
\begin{align}\label{2.57}
&\sup_{0\le t\le T}\big[\bar{\rho}\big(\|\nabla b\|_{L^2}^2+\|b\|_{L^4}^4+\|\sqrt{\rho}E\|_{L^2}^2\big)+\|\nabla u\|_{L^2}^2\big]
+\int_0^T\Big\|\Big(\sqrt{\rho}u_t, \frac{\nabla F}{\sqrt{\bar{\rho}}},
\frac{\nabla w}{\sqrt{\bar{\rho}}}\Big)\Big\|_{L^2}^2dt\nonumber\\
&\quad+C\bar{\rho}\int_0^T\big(\|b_t\|_{L^2}^2+\|\nabla^2b\|_{L^2}^2
+\||b||\nabla b|\|_{L^2}^2+\||u||\nabla u|\|_{L^2}^2+\|\nabla\theta\|_{L^2}^2\big)dt\nonumber\\
&\le C\|\nabla u_0\|_{L^2}^2+C\bar{\rho}\big(\|\sqrt{\rho_0}E_0\|_{L^2}^2+\|\nabla b_0\|_{L^2}^2\big).
\end{align}
This implies \eqref{2.47}.

4. We deduce from \eqref{2.2} and \eqref{2.57} that
\begin{align}\label{2.58}
&\sup_{0\le t\le T}\big(\|\sqrt{\rho}u\|_{L^2}^2+\|b\|_{L^2}^2\big)+\int_0^T\big(\|\nabla u\|_{L^2}^2+\|\nabla b\|_{L^2}^2\big)dt\nonumber\\
&\le \|\sqrt{\rho_0}u_0\|_{L^2}^2+\|b_0\|_{L^2}^2
+C\sup_{0\le t\le T}\|\rho\|_{L^3}^2\int_0^T\|\nabla\theta\|_{L^2}^2dt\nonumber\\
&\le \|\sqrt{\rho_0}u_0\|_{L^2}^2+\|b_0\|_{L^2}^2+C\frac{1}{\bar{\rho}}
\sup_{0\le t\le T}\|\rho\|_{L^3}^2\big[\bar{\rho}(\|\sqrt{\rho_0}E_0\|_{L^2}^2+\|\nabla b_0\|_{L^2}^2)+\|\nabla u_0\|_{L^2}^2\big]\nonumber\\
&\le \|\sqrt{\rho_0}u_0\|_{L^2}^2+\|b_0\|_{L^2}^2
+C\eta^\frac12\frac{1}{\bar{\rho}^2}\sup_{0\le t\le T}\|\rho\|_{L^3}.
\end{align}
Indeed, it follows from \eqref{2.14}, \eqref{2.44}, \eqref{2.58}, and Young's inequality that
\begin{align}\label{2.59}
&\sup_{0\le t\le T}\|\rho\|_{L^3}^3+\int_0^T\int\rho^3pdxdt\nonumber\\
&\le C\|\rho_0\|_{L^3}^3+C\sup_{0\le t\le T}
\Big(\|\rho\|_{L^\infty}^\frac23\|\sqrt{\rho}u\|_{L^2}^\frac13
\|\sqrt{\rho}E\|_{L^2}^\frac13\Big)\|\rho\|_{L^3}^3
+C\bar{\rho}^2\sup_{0\le t\le T}\|\rho\|_{L^3}^2\int_0^T\big(\|\nabla u\|_{L^2}^2+\|\nabla b\|_{L^2}^2\big)dt\nonumber\\
&\le C\|\rho_0\|_{L^3}^3+C\big(\eta^\frac{1}{12}+\eta^\frac12\big)\sup_{0\le t\le T}\|\rho\|_{L^3}^3
+C\bar{\rho}^2\sup_{0\le t\le T}\|\rho\|_{L^3}^2\big(\|\sqrt{\rho_0}u_0\|_{L^2}^2+\|b_0\|_{L^2}^2\big)\nonumber\\
&\le C\|\rho_0\|_{L^3}^3+C\Big(\eta^\frac{1}{2}+\frac14\Big)\sup_{0\le t\le T}\|\rho\|_{L^3}^3
+C\bar{\rho}^6\big(\|\sqrt{\rho_0}u_0\|_{L^2}^2+\|b_0\|_{L^2}^2\big)^3,
\end{align}
which implies \eqref{2.45} by choosing $\eta_0$ suitably small.

Combining \eqref{2.58} and \eqref{2.59}, we get that
\begin{align*}
&\sup_{0\le t\le T}\bar{\rho}^2(\|\sqrt{\rho}u\|_{L^2}^2+\|b\|_{L^2}^2)
+\bar{\rho}^2\int_0^T(\|\nabla u\|_{L^2}^2+\|\nabla b\|_{L^2}^2)dt\nonumber\\
&\le C\bar{\rho}^2(\|\sqrt{\rho_0}u_0\|_{L^2}^2+\|b_0\|_{L^2}^2)+C\eta^\frac12\sup_{0\le t\le T}\|\rho\|_{L^3}\nonumber\\
&\le C(\|\rho_0\|_{L^3}+\bar{\rho}^2(\|\sqrt{\rho_0}u_0\|_{L^2}^2+\|b_0\|_{L^2}^2)),
\end{align*}
which gives \eqref{2.46}.

Finally, \eqref{2.48} follows from Lemma \ref{l27}, \eqref{2.46}, and \eqref{2.47}.
\end{proof}

\begin{proposition}\label{pro}
Assume that $3\mu>\lambda$. Let $\eta_0$, $N_T$, and $N_0$ be as in Lemma \ref{l28}. Then, the followings hold true.

(i) There exists a number $\varepsilon_0\in (0, \eta_0)$ depending only on $R$, $\mu$, and $\lambda$ such that if
\begin{align}\label{3.54}
\sup_{0\le t\le T}\|\rho\|_{L^\infty}\le 4\bar{\rho}, \ N_T\le \sqrt{\varepsilon_0}, \ N_0\le \varepsilon_0,
\end{align}
then
\begin{align*}
\sup_{0\le t\le T}\|\rho\|_{L^\infty}\le 2\bar{\rho}, \ N_T\le\frac{ \sqrt{\varepsilon_0}}{2}.
\end{align*}

(ii) As a consequence of (i), the following estimates hold
\begin{align}\label{3.56}
\sup_{0\le t\le T}\|\rho\|_{L^\infty}\le 2\bar{\rho}, \ N_T\le\frac{ \sqrt{\varepsilon_0}}{2},
\end{align}
provided that $N_0\le \varepsilon_0$ is sufficiently small.
\end{proposition}
\begin{proof}[Proof]
(i) By \eqref{3.54}, all the conditions in Lemma \ref{l28} hold true when $\varepsilon_0\le \eta_0$ is small enough. Thus, we have
\begin{align*}
N_T&\le \bar{\rho}(\|\rho_0\|_{L^3}+\bar{\rho}^2(\|\sqrt{\rho_0}u_0\|_{L^2}^2+\|b_0\|_{L^2}^2))
(\|\nabla u_0\|_{L^2}^2+\bar{\rho}(\|\sqrt{\rho_0}E_0\|_{L^2}^2+\|\nabla b_0\|_{L^2}^2))\le C\varepsilon_0\le
\frac{\sqrt{\varepsilon_0}}{2}.
\end{align*}
and
\begin{align*}
\sup_{0\le t\le T}\|\rho\|_{L^\infty}\le C\bar{\rho}e^{CN_0^\frac16+CN_0^\frac12}\le
\bar{\rho}e^{C\varepsilon_0^\frac16+C\varepsilon_0^\frac12}\le 2\bar{\rho},
\end{align*}
provided that $\varepsilon_0$ is sufficiently small.

(ii) Define
\begin{align*}
T^{\#}\triangleq\max\Big\{T'\in (0, T]\Big| \sup_{0\le t\le T'}\|\rho\|_{L^\infty}\le 2\bar{\rho},~~
N_{T'}\le \sqrt{\varepsilon_0}\Big\}.
\end{align*}
Then, by (i), we have
\begin{align}\label{fhf}
\sup_{0\le t\le T}\|\rho\|_{L^\infty}\le 2\bar{\rho}, \quad N_{T'}\le \frac{\sqrt{\varepsilon_0}}{2},\quad \forall T'\in(0, T^{\#}).
\end{align}

If $T^{\#}<T$, noticing that $N_{T'}$, and $\sup\limits_{0\le t\le T'}\|\rho\|_{L^\infty}$
are continuous on $[0, T]$, there is another time $T^{\#\#}\in (T^{\#}, T]$ such that
\begin{align*}
\sup_{0\le t\le T'}\|\rho\|_{L^\infty}\le 2\bar{\rho},~~
N_{T'}\le \sqrt{\varepsilon_0},
\end{align*}
which contradicts to the definition of $T^{\#}$. Thus, we have $T^{\#}=T$, and \eqref{3.56}
follows from \eqref{fhf} and the continuity of
$N_{T'}$ and $\sup\limits_{0\le t\le T'}\|\rho\|_{L^\infty}$ on $[0, T]$.

\subsection{Proof of Theorem \ref{thm1}}

With all the {\it a priori} estimates established in Section 2, we can immediately obtain the existence result of Theorem \ref{thm1} by standard arguments as those in \cite{LZ20}. Here we omit the details for simplicity.
\end{proof}

\section*{Acknowledgments}
The authors would like to express their gratitude to the reviewers for careful reading and helpful suggestions which led to an improvement of the original manuscript.


\begin{thebibliography}{10}



\bibitem{D2016}
E. DiBenedetto, \textit{Real analysis, 2nd edition}, Birkh{\"a}user, New York, 2016.


\bibitem{DF2006}
B. Ducomet and E. Feireisl, The equations of magnetohydrodynamics: on the interaction between matter and radiation in the evolution of gaseous stars, Comm. Math. Phys., 266 (2006), 595--629.


\bibitem{FY09}
J. Fan and W. Yu, Strong solution to the compressible magnetohydrodynamic equations with vacuum, Nonlinear Anal. Real World Appl., 10 (2009), no. 1, 392--409.


\bibitem{F2004}
E. Feireisl, \textit{Dynamics of viscous compressible fluids}, Oxford University Press, Oxford, 2004.


\bibitem{FL2020}
E. Feireisl and Y. Li, On global-in-time weak solutions to the magnetohydrodynamic system of compressible inviscid fluids, Nonlinearity, 33 (2020), 139--155.


\bibitem{HHPZ}
G. Hong, X. Hou, H. Peng, and C. Zhu, Global existence for a class of large solutions to three dimensional compressible magnetohydrodynamic equations with vacuum, SIAM J. Math. Anal., 49 (2017), no. 4, 2409--2441.


\bibitem{HJP22}
X. Hou, M. Jiang, and H. Peng, Global strong solution to 3D full compressible magnetohydrodynamic flows with vacuum at infinity, Z. Angew. Math. Phys., 73 (2022), no. 1, Paper No. 13.


\bibitem{HW08}
X. Hu and D. Wang, Global solutions to the three-dimensional full compressible magnetohydrodynamic flows, Commun. Math. Phys., 283 (2008), no. 1, 255--284.


\bibitem{HW10}
X. Hu and D. Wang, Global existence and large-time behavior of solutions to the three dimensional equations of compressible magnetohydrodynamic flows, Arch. Ration. Mech. Anal., 197 (2010),  no. 1, 203--238.


\bibitem{HL18}
X. Huang and J. Li, Global classical and weak solutions to the three-dimensional full compressible Navier-Stokes system with vacuum and large oscillations, Arch. Ration. Mech. Anal., 227 (2018), no. 3, 995--1059.


\bibitem{LXZ13}
H. Li, X. Xu, and J. Zhang, Global classical solutions to 3D compressible magnetohydrodynamic equations with large oscillations and vacuum, SIAM J. Math. Anal., 45 (2013), no. 3, 1356--1387.


\bibitem{L20}
J. Li, Global small solutions of heat conductive compressible Navier-Stokes equations with vaccum: smallness on scaling invariant quantity,
Arch. Ration. Mech. Anal., 237 (2020), no. 2, 899--919.


\bibitem{LQ2012}
T. Li and T. Qin, \textit{Physics and partial differential equations. vol. 1}, Translated from the Chinese original by Yachun Li, Higher Education Press, Beijing, 2012.


\bibitem{LG2014}
X. Li and B. Guo, On the equations of thermally radiative magnetohydrodynamics, J. Differential Equations, 257 (2014), no. 9, 3334--3381.


\bibitem{LS2021}
Y. Li and Y. Sun, On global-in-time weak solutions to a two-dimensional full compressible non-resistive MHD system, SIAM J. Math. Anal., 53 (2021), no. 4, 4142--4177.


\bibitem{L21}
Z. Liang, Global strong solutions of Navier-Stokes equations for heat-conducting compressible fluids with vacuum at infinity,
J. Math. Fluid Mech., 23 (2021), no. 1, Paper No. 17.


%\bibitem{L1998}
%P. L. Lions, \textit{Mathematical topics in fluid mechanics, vol. II: compressible models}, Oxford University Press, Oxford,
%1998.


\bibitem{LZ20}
Y. Liu and X. Zhong, Global well-posedness to three-dimensional full compressible magnetohydrodynamic equations with vacuum, Z. Angew. Math. Phys.,
71 (2020), no. 6, Paper No. 188.


\bibitem{LZ21}
Y. Liu and X. Zhong, Global existence and decay estimates of strong solutions for compressible non-isentropic magnetohydrodynamic flows with vacuum, https://arxiv.org/abs/2108.06726.


\bibitem{LSX16}
B. L{\"u}, X. Shi, and X. Xu, Global existence and large-time asymptotic behavior of strong solutions to the compressible magnetohydrodynamic equations with vacuum, Indiana Univ. Math. J., 65 (2016), 925--975.


\bibitem{MN80}
A. Matsumura and T. Nishida,
The initial value problem for the equations of motion of viscous and heat-conductive gases, J. Math. Kyoto Univ., 20 (1980), no. 1, 67--104.


\bibitem{MN83}
A. Matsumura and T. Nishida,
Initial-boundary value problems for the equations of motion of compressible viscous and heat-conductive fluids, Comm. Math. Phys., 89 (1983), no. 4, 445--464.


\bibitem{SH12}
A. Suen and D. Hoff, Global low-energy weak solutions of the equations of three-dimensional compressible magnetohydrodynamics, Arch. Ration. Mech. Anal., 205 (2012), 27--58.


\bibitem{WZ17}
H. Wen and C. Zhu, Global solutions to the three-dimensional full compressible Navier-Stokes equations with vacuum at infinity in some classes of large data, SIAM J. Math. Anal., 49 (2017), no. 1, 162--221.

\bibitem{WW17}
J. Wu and Y. Wu, Global small solutions to the compressible 2D magnetohydrodynamic system without magnetic diffusion, Adv. Math., 310 (2017), 759--888.



\end{thebibliography}
\end{document}